\documentclass{article}

\usepackage{multicol}
\usepackage{color}
\usepackage{amsthm,amssymb,amsfonts,mathrsfs}
\usepackage{amsmath}

\newcommand{\RN}[1]{%
	\textup{\uppercase\expandafter{\romannumeral#1}}%
}

\newtheorem{Theorem}{Theorem}[section]
\newtheorem{Proposition}{Proposition}[section]
\newtheorem{Lemma}{Lemma}[section]

\theoremstyle{definition}

\newtheorem{Definition}{Definition}[section]
\newtheorem{Remark}{Remark}

\newtheorem{Assumptions}{Hypothesis}[section]

\def\R{{\mathbb{R}}}

\def\cU{{\mathcal{U}}}

\def\t\cU{{\widetilde{{\mathcal{U}}}}}

\newcommand\norm[1]{\left\lVert#1\right\rVert}

\def\ds{\displaystyle}

\title{Degenerate fourth order parabolic equations with Neumann boundary conditions}
\author{{\sc Alessandro Camasta}\thanks{The author is a member of the  {\it Gruppo Nazionale per l'Analisi Ma\-te\-matica, la Probabilit\`a e le loro Applicazioni (GNAMPA)} of the Istituto Nazionale di Alta Matematica (INdAM), a member of {\it UMI ``Modellistica Socio-Epidemiologica (MSE)''} and he is partially supported by PRIN 2017-2019 {\it Qualitative and quantitative aspects of nonlinear PDEs.} He is also supported by the project {\it Mathematical models for interacting dynamics on networks (MAT-DYN-NET)
				CA18232.}}\\
Department of Mathematics\\ University of Bari Aldo Moro\\
Via
E. Orabona 4\\ 70125 Bari - Italy\\ e-mail: alessandro.camasta@uniba.it\\
{\sc Genni Fragnelli}\thanks{The author is a member of the  {\it Gruppo Nazionale per l'Analisi Ma\-te\-matica, la Probabilit\`a e le loro Applicazioni (GNAMPA)} of the Istituto Nazionale di Alta Matematica (INdAM), a member of {\it UMI ``Modellistica Socio-Epidemiologica (MSE)''} and is supported by FFABR {\it Fondo per il finanziamento delle attivit\`a base di ricerca} 2017, by  PRIN 2017-2019 {\it Qualitative and quantitative aspects of nonlinear PDEs} and by the
	HORIZON$_-$EU$_-$DM737 project 2022 {\it COntrollability of PDEs in the Applied Sciences (COPS)} at Tuscia University. She is also supported by the project {\it Mathematical models for interacting dynamics on networks (MAT-DYN-NET)
		CA18232.}}\\
Department of Ecology and Biology\\ Tuscia University\\
Largo dell'Universit\`a, 01100 Viterbo - Italy\\ e-mail: genni.fragnelli@unitus.it}

\date{}
\begin{document}

\maketitle

\abstract{We study the generation property for a fourth order operator in divergence or in non divergence form with suitable Neumann boundary conditions. As a consequence we obtain the well posedness for the parabolic equations governed by these operators. The novelty of this paper is that the operators depend on a function $a: [0,1] \rightarrow \R_+$ that degenerates somewhere in the interval.}
\vspace{0.3cm}

\vspace{0.03cm}

\noindent 2000AMS Subject Classification: Primary: 47D06, 35K65; Secondary: 47B25, 47N20

\section{Introduction}
In this paper we provide a full analysis of two  operators under certain Neumann boundary conditions. More precisely, we consider
\begin{equation}\label{operatori A_i}
	A_iu:=\begin{cases}
		\frac{\partial^2}{\partial x^2}\Bigl (a\frac{\partial^2u}{\partial x^2}\Bigr ) &\text{if }\,\, i=1,\\
		a\frac{\partial^4u}{\partial x^4} &\text{if }\,\, i=2,
	\end{cases}
\end{equation}
where the function $a:[0,1]\to\mathbb{R}_+$ degenerates at $x_0\in [0,1]$.

We are interested in studying this kind of degenerate operators, which represent the subject of numerous papers and books, together with suitable Neumann boundary conditions for two reasons. 

First of all, many problems coming from Physics, Biology and Economics are governed by degenerate operators
that often have to do with Neumann boundary conditions. 

Secondly, there are real problems connected to freely supported beams (with only fixed end points and without any other condition) and beams supported at only one end, that lead to natural boundary conditions. It is relevant to observe that end point conditions are modular, i.e., it is possible to use different end point conditions at each end of the beam. In particular, the main applications are characterized by the choice of different conditions:
\begin{itemize}
	\item clamped: specifies the position and the derivative;
	\item freely supported: specifies the position (in this case the natural boundary condition is that the second derivative is zero at the end point);
	\item no condition: neither position nor end point are specified (the natural boundary conditions fix the second and third derivatives at the end point to be zero).
\end{itemize}

Let us present some interesting contributions about the boundary conditions for problems associated to the beam analysis.

Starting from a static fourth order Euler-Bernoulli beam equation, we recall \cite{bo yang} where the authors study the fourth order ordinary differential equation
\begin{equation}\label{+}
	u''''(x)=g(x)f(u(x)),\quad\,\,\,\,\,\,\,\,\,\,\,0\le x\le 1,
\end{equation}
which is often referred to as the beam equation. In the formulation of the previous equation, the authors assume that $f:[0,+\infty)\to [0,+\infty)$ and $g:[0,1]\to [0,+\infty)$ are continuous functions and $g$ is such that $\int_0^1 g(t)dt>0$. From the physical point of view, \eqref{+} describes the deflection or deformation of an elastic beam under a certain force, together with the following  boundary conditions
\begin{equation}\label{bound cond}
	u(0)=u'(0)=u''(1)=u'''(1)=0,
\end{equation}
which arise from the study of elasticity.
 These conditions have definite physical meanings; in fact, they mean that the
beam is embedded at the end $x=0$ and free at the end $x=1$. 
Furthermore, the boundary conditions (\ref{bound cond}) represent a special case of the so called right focal boundary conditions, for which an extensive research has been done (we refer to  \cite{argai} for a systematic survey of this field; for some recent results on focal boundary value problems, see, e.g., \cite{argai regan} and the references therein).
Equation \eqref{+} is also analyzed in \cite{bo yang2}, where it is coupled with the boundary conditions
\begin{equation*}
	u(0)=u'(0)=u'(1)=u'''(1)=0,
\end{equation*}
which mean that the beam is embedded at $x=0$ and fastened with a sliding clamp at  $x=1$.
In \cite{GUPTA} a boundary value problem under the following boundary conditions is considered:
\begin{equation*}
	u'(0)=u'''(0)=u'''(\pi)=u'(\pi)=0.
\end{equation*}
It means that the beam is fastened with sliding clamps at both ends $x=0$ and $x=\pi$. Similar conditions are considered in \cite{yang zhang} in order to study the existence of non trivial solutions to the semilinear fourth order problem
\begin{equation*}
	\begin{cases}
		u''''(x)-2u''(x)+u(x)=f(x,u(x)),& 0<x<1,\\
		u'(0)=u'(1)=u'''(0)=u'''(1)=0,
	\end{cases}
\end{equation*}
being $f\in\mathcal{C}([0,1]\times\mathbb{R};\mathbb{R})$.
Imposing some ideal conditions, the deflection of a beam rigidly 
fastened at left and simply fastened at right led \cite{argai0} to a fourth order non linear differential equation together with two point boundary conditions.
Then \cite{elgindi} takes in consideration different boundary conditions, corresponding to various ways in which the ends of a beam may be supported, and examines a class of  fourth order non linear boundary value problems which govern the equilibrium states of a beam-column. In this case the source of the non linearity comes from a non linear lateral constraint (foundation) and the equilibrium equation is formulated as a fourth order non linear differential equation along with one of the following six sets of boundary conditions:
\begin{equation*}
	\begin{split}
		u(0)=u''(0)=u(1)=u''(1)=0,\\
		u(0)=u''(0)=u(1)=u'(1)=0,\\
		u(0)=u''(0)=u'(1)=u'''(1)=0,\\
		u(0)=u'(0)=u(1)=u'(1)=0,\\
		u(0)=u'(0)=u''(1)=u'''(1)=0,\\
		u(0)=u'(0)=u'(1)=u'''(1)=0.
	\end{split}
\end{equation*}
These conditions represent, respectively, the following situations: both ends are simply-supported; one end is simply-supported and the other is fixed; one end is simply-supported and the other is sliding clamped; both ends are fixed; one end is fixed and the other is free and one end is fixed and the other is sliding clamped. See also \cite{kosmatov}, in which the beam equation is studied together with the boundary conditions
\[
	u(0)=u'(0)=u'(1)=u(1)=0,
\]
which mean that the beam is embedded at both ends $x=0$ and $x=1$.
For other works connected to boundary value problems of the beam equation we refer, for example, to \cite{bai}, \cite{davis}, 
\cite{graef}, \cite{gupta},
\cite{liu}, 
\cite{ma2}, \cite{yang}, \cite{yao2} and the references therein.

However,  there are some equations of great interest due to their applications in many engineering fields, such as mechanical engineering and civil engineering, that are not written as a static beam equation, but as a partial differential equation governed by the operators introduced in \eqref{operatori A_i}. In this context, the coefficient $a(x)$ may be thought as the flexural rigidity (for $i=1$) or as the density of the beam (for $i=2$).
Indeed, the problem of the transversely vibrating beam, for which the most important factor is represented by the bending effect, is formulated in terms of the partial differential equation of motion, of the boundary and initial conditions, which give rise to an initial boundary value problem.
More precisely, if $u(t,x)$ denotes the transverse displacement at time $t$ and position $x$ from one
end of the beam taken as the origin, $a(x)$ the flexural rigidity,
and $m(x) > 0$ the lineal mass, the transverse motion of an unloaded thin beam is represented by the following Euler-Bernoulli beam equation
\[
m(x)u_{tt} +(a(x)u_{xx})_{xx} =0,
\]
where $t>0$ and $x \in (0,L)$.
On the other hand, if we consider an external forcing function $f$ and if $E$ is the elastic modulus (or Young's modulus), $I$ is the area of inertia and $\lambda$ is the mass per unit length, then 
 the problem of the deflection of the beam can be described by the following fourth order partial differential equation:
	\begin{equation}\label{01}
		\lambda \frac{\partial^2u}{\partial t^2}+\frac{\partial^2}{\partial x^2}\Biggl (EI\frac{\partial^2u}{\partial x^2}\Biggr )=f(u).
	\end{equation}
	The applied load function $f$, which is  sufficiently smooth, can depend on $u$, as in this case, or  on time and position.
If $E, I, \lambda$ are constant, after a rescaling, the previous equation becomes
\begin{equation*}
	\frac{\partial^2u}{\partial t^2}+ \frac{\partial^4u}{\partial x^4}=f(u).
\end{equation*}
Starting from \eqref{01}, a lot of papers are devoted to the study of models coming from the previous one. In particular, in \cite{russell1} and \cite{russell2} Russell notes that  viscous damping models such as
	\begin{equation*}
		\rho\frac{\partial^2u}{\partial t^2}+2\frac{\partial u}{\partial t}+	\frac{\partial^2}{\partial x^2}\Biggl (EI	\frac{\partial^2u}{\partial x^2} \Biggr )=0,
	\end{equation*}
	which produce uniform damping rates, are inadequate if experimentally observed damping
properties are to be incorporated in the model. At the end of the
nineteenth century, Kelvin and Voigt note  that damping rates tend to increase with frequency. Incorporated into
the Euler-Bernoulli beam model their approach yields an equation of the form
	\begin{equation*}
		\rho\frac{\partial^2u}{\partial t^2}+2\rho\frac{\partial^3}{\partial t\partial x^2}\Biggl (EI	\frac{\partial^2u}{\partial x^2} \Biggr )+	\frac{\partial^2}{\partial x^2}\Biggl (EI	\frac{\partial^2u}{\partial x^2} \Biggr )=0.
	\end{equation*}
An interesting variant of a Kelvin-Voigt viscoelastic damped Euler-Bernoulli beam is represented by the following equation
\begin{equation}\label{wp}
	\frac{\partial^2u}{\partial t^2}+\frac{\partial^2}{\partial x^2}\Biggl (a\frac{\partial^3u}{\partial t\partial x^2}\Biggr )+\frac{\partial^2}{\partial x^2}\Biggl (a	\frac{\partial^2u}{\partial x^2}\Biggr )=0,
\end{equation}
for which we appreciate a deep connection with parabolic equations. Indeed, considering $f$ such that
\begin{equation*}
	\frac{\partial f}{\partial t}+\frac{\partial^2}{\partial x^2}\Biggl (a	\frac{\partial^2u}{\partial x^2}\Biggr )=0,
\end{equation*}
equation \eqref{wp} can be rewritten as
\begin{equation}\label{problem1}
	\frac{\partial u}{\partial t}+\frac{\partial^2}{\partial x^2}\Biggl (a	\frac{\partial^2u}{\partial x^2}\Biggr )=f.
\end{equation}
A problem similar to \eqref{problem1} can also be found in the study of thin films and phase field models, see \cite{elliott}. Indeed, in this case the general form is 
\[
u_t= \nabla\cdot(f(u)\nabla w),
\]
where
\[
w= -\gamma \Delta u + \varphi(u),
\]
$f$ is non negative, $u>0$ and both $f$ and $\varphi$ are smooth. If we take $\gamma =1$ and after some computations, one can rewrite the previous problem as \eqref{problem1} (for more details see \cite{bertozzi}).

In this respect \cite{CF} and \cite{3} are the first papers that deal with 
degenerate fourth order operators. In particular, in \cite{CF} we consider the operator $A_2$ defined in \eqref{operatori A_i}, while in   \cite{3} the operator $A_1$ is studied, both
with Dirichlet boundary conditions. We underline that while in \cite{CF} the degeneracy point belongs to $[0,1]$, in \cite{3} only the interior degenerate case is considered.
Moreover, in \cite{CF} the properties of $A_2$ are used to obtain the well posedness of the associated parabolic problem.

Motivated by the above problems, in this paper we study  the operators $A_i$, $i=1,2$, introduced in \eqref{operatori A_i} with Neumann boundary conditions. We admit two types of degeneracy for $a$, namely weak and strong degeneracy according to the following definitions:
\begin{Definition}
	The function $a$ is weakly degenerate if there exists a point $x_0\in [0,1]$ such that  $a(x_0)=0$, $a>0$ on $[0,1]\setminus \{x_0\}$, $a\in\mathcal{C}[0,1] \cap \mathcal{C}^1([0,1] \setminus \{x_0\})$ and $\frac{1}{a}\in L^1(0,1)$.
\end{Definition}
For example, as $a$, we can consider $a(x)=|x-x_0|^{\alpha}$, $0<\alpha<1$.
\begin{Definition}\label{def strong function}
	The function $a$ is strongly degenerate if there exists a point $x_0\in [0,1]$ such that  $a(x_0)=0$, $a>0$ on $[0,1]\setminus \{x_0\}$, $a\in \mathcal{C}^1([0,1] \setminus \{x_0\}) \cap W^{1, \infty}(0,1)$ and $\frac{1}{a}\notin L^1(0,1)$.
\end{Definition}
For example, as $a$, we can consider $a(x)=|x-x_0|^{\alpha}$, $\alpha\ge 1$.

Regarding the boundary terms considered in this paper, we underline that they are different according to the degeneracy of $a$ and to the {\it form} of the operator. Consider, for example, the operator  in divergence form $A_1$. If the degeneracy point $x_0$ belongs to $(0,1)$, we can consider the following boundary conditions
\begin{equation}\label{condition1}
\begin{cases}
u''(0)=u'''(0)=0,\\
u''(1)=u'''(1)=0.
\end{cases}
\end{equation}
Indeed, since $u$ is $\mathcal C^3$ in a neighbourhood of $x=0$ and $x=1$, the previous terms are well defined. If $x_0=0$, then at $x=1$ we can consider the same conditions as in \eqref{condition1}, but at $0$ we have to require other conditions since we do not know that $u \in \mathcal C^3$  at $x=0$. Actually, the natural boundary conditions at $x=0$ are
\[
(au'')(0)=(au'')'(0)=0,
\]
which make sense since we assume that $au'' \in H^2(0,1)$.
Analogously, if $x_0=1$. On the other hand, take the operator in non divergence form $A_2$; then the boundary conditions are given by \eqref{condition1}  whether the function $a$ is weakly or strongly degenerate and $x_0 \in (0,1)$. Actually \eqref{condition1} are the natural boundary conditions even if $x_0=0$ or $x_0=1$ and the function $a$ is weakly degenerate; on the other hand, in the strongly degenerate case, we require
\[
\begin{cases}
u''(0)=u'''(0)=u''(1)=u(1)=0, & \text{if} \; x_0=1,\\
u''(1)=u'''(1)=u''(0)=u(0)=0, & \text{if} \; x_0=0,
\end{cases}
\]
(see Sections \ref{Sec2} and \ref{Sec3} for more details).

The paper is constituted by two main sections in which we study  the operators in divergence and non divergence form with suitable boundary conditions. We then apply the properties of the operators $A_i$, $i=1,2$, to study the well posedness of the following parabolic problems
\[
	\begin{cases}
		\frac{\partial u}{\partial t}(t,x)+A_iu(t,x)=h(t,x), &(t,x)\in (0,T) \times (0,1),\\
		u(0,x)=u_0(x),&x\in(0,1),
	\end{cases}
\]
where  $T>0$, $u_0$ belongs to a suitable Hilbert space $H$ and $h\in L^2(0,T;H)$.  In particular, $H$ denotes the Hilbert space $L^2(0,1)$, in the divergence form (i.e., $i=1$), and a suitable weighted Hilbert space, in the non divergence one (i.e., $i=2$).

A final comment on the notations: by $C$ we shall denote universal positive constants, which are allowed to vary from line to line; we denote with $u'$ the derivative of a function $u$ depending only on
one variable x, which we assume to vary in $[0, 1]$, while $u_t$, $u_{x}$ are the usual partial derivatives $\frac{\partial u}{\partial t}$, $\frac{\partial u}{\partial x}$, respectively, of a function $u=u(t,x)$.

It is worth noting that in the present work we deal with real function spaces, but the assertions can be easily extended to the complex case.

\section{The operator in divergence form}\label{Sec2}
In this section we consider the operator $A_1u:= (au_{xx})_{xx}$
with suitable boundary conditions and we prove its generation property.

To this aim we discuss two situations: the weakly degenerate case and the strongly degenerate one.

\subsection{Weakly degenerate setting}\label{Section 2}
Throughout this subsection we assume that the function $a$ is weakly degenerate. 

In order to deal with the generation property of  $A_1$ we have to consider a suitable domain that will be defined later. To this aim,
inspired by \cite{ACF}, \cite{FGGR} or \cite{fm2013}, we introduce the following (weighted) Hilbert spaces:
\[\begin{aligned}
	H^i_a(0,1):=\{u\in H^{i-1}(0,1):& \; u^{(i-1)} \text{ is absolutely continuous in [0,1]},\\
	& \sqrt{a}u^{(i)}\in L^2(0,1)\},
\end{aligned}\]
endowed with the norms
\begin{equation}\label{normadiv}
	\|u\|^2_{H^i_a(0,1)}:= \sum_{j=0}^{i-1}\|u^{(j)}\|^2_{L^2(0,1)}+ \|\sqrt{a}u^{(i)}\|^2_{L^2(0,1)} \quad \forall \; u \in H^i_a(0,1),
\end{equation}
$i=1,2$. Here  $H^0(0,1):= L^2(0,1)$ and $u^{(0)}=u$. Observe that $H^1_a(0,1)$ is exactly the space considered for the first time in \cite{CMP}.
Furthermore, in \cite[Proposition 2.1]{CF-articolo Wentzell} it is proved that for every $u\in H^i_a(0,1)$ the norm $\|u\|^2_{H^i_{a}(0,1)}$ is equivalent to the following one
\begin{equation}\label{norma tre stanghette}
|||u|||^2:= \|u\|^2_{L^2(0,1)} +  \|\sqrt{a}u^{(i)}\|^2_{L^2(0,1)},
\end{equation}
$i=1,2$.
Then, introducing
\begin{equation*}
	{\mathcal Z}(0,1):=\{u\in H^2_a(0,1): au''\in H^2(0,1)\},
\end{equation*}
define the operator $A_1$ by 
\begin{equation}\label{ope1}
	\begin{split}
		A_1u:&=(au'')'',\\
		\forall\; u\in D(A_1):=\{u\in {\mathcal Z}(0,1):\, u''(0)&=u''(1)=0, u'''(0)=u'''(1)=0\},
	\end{split}
\end{equation}
if $x_0\in (0,1)$.
If $x_0\in \{0,1\}$, we consider natural boundary conditions which arise by computation (see Theorem \ref{teorema generazione (div)}). Indeed, the definition of the domain of the operator is slightly different.
More precisely, if $x_0=0$, we consider  as $D(A_1)$ the set
\begin{equation*}
	\begin{split}
		D(A_1):=\{u\in {\mathcal Z}(0,1): u''(1)=u'''(1)=0, (au'')(0)=(au'')'(0)=0\},
	\end{split}
\end{equation*}
while, if $x_0=1$, the domain becomes
\begin{equation*}
	\begin{split}
		D(A_1):=\{u\in {\mathcal Z}(0,1): u''(0)=u'''(0)=0, (au'')(1)=(au'')'(1)=0\}.
	\end{split}
\end{equation*}
Note that in \cite{ACF} a degenerate operator of a second order in divergence form with a Neumann boundary condition of the type $(au_x)(t,0)=0$ and $u(t,1)=0$ is considered  if $x_0=0$; we refer to \cite{BFM} if $x_0 \in (0,1)$.

To prove that $(A_1, D(A_1))$ generates a strongly continuous semigroup the following formula of integration by parts is essential:
\begin{Lemma}\label{lemma gauss-green}
	For all $(u,v)\in {\mathcal Z}(0,1) \times H^2_a(0,1)$ one has
	\begin{equation}\label{gf0}
		\int_{0}^{1}(au'')''v\,dx=[(au'')'v]^{x=1}_{x=0}-[au''v']^{x=1}_{x=0}+\int_{0}^{1}au''v''dx.
	\end{equation}
\end{Lemma}
\begin{proof}
	Take $u \in {\mathcal Z}(0,1)$ and $v \in H^2_a(0,1)$. Clearly, $au'', (au'')' \in H^1(0,1)$ and $v, v'$ are absolutely continuous;
	hence
	\[
	\begin{aligned}
		\int_0^1 (au'')'' v\,dx&=  [(au'')'v]_{x=0}^{x=1}-\int_0^1 (au'')'v'dx\\
		&=  [(au'')'v]_{x=0}^{x=1}-[au''v']_{x=0}^{x=1}+\int_0^1 au''v''dx.
	\end{aligned}
	\]
	Observe that $au''v''  \in L^1(0,1)$ since, by assumption,  $\sqrt{a}u''$ and $\sqrt{a}v''$ belong to $L^2(0,1)$.
\end{proof}

We underline that in order to prove Lemma \ref{lemma gauss-green}, we do not use the assumption $\frac{1}{a} \in L^1(0,1)$. It is sufficient to require $a\in \mathcal C[0,1]$.

\begin{Remark}\label{remark 1}
	Note that:
	\begin{enumerate}
		\item[1)] if $u\in D(A_1)$ and $v\in H^2_a(0,1)$, then the boundary terms $[(au'')'v]^{x=1}_{x=0}$ and $[au''v']^{x=1}_{x=0}$ are equal to $0$ and (\ref{gf0}) becomes
		\begin{equation}\label{gf0 senza bordo}
			\int_{0}^{1}(au'')''v\,dx=\int_{0}^{1}au''v''dx;
		\end{equation}
		\item[2)] the boundary conditions prescribed in the domain $D(A_1)$ play a crucial role to prove the generation property of the operator. Indeed, the proof of the next theorem is based on  (\ref{gf0 senza bordo}). 
	\end{enumerate}
\end{Remark}

The next result holds.
\begin{Theorem}\label{teorema generazione (div)}
	The operator $A_1:D(A_1)\to L^2(0,1)$ is self-adjoint and non negative on $L^2(0,1)$ with dense domain. Hence $-A_1$ generates a contractive analytic semigroup of angle $\frac{\pi}{2}$ on $L^2(0,1)$.
\end{Theorem}
\begin{proof}
	In order to show that $A_1$ is non negative and self-adjoint it is sufficient to prove that $A_1$ is symmetric, non negative and $(I+A_1)(D(A_1))=L^2(0,1)$. Indeed, if $A_1$ is non negative and $I+A_1$ is surjective on $D(A_1)$, then $A_1$ is maximal monotone and in this case $A_1$ is symmetric if and only if $A_1$ is self-adjoint. Moreover, by  \cite[Corollary 3.20]{en}, we have that
	$-A_1$ is densely defined and generates a contraction semigroup.
	
	\underline{$A_1$ is symmetric}: by (\ref{gf0 senza bordo}), for any $u,v\in D(A_1)$, one has
	\begin{equation*}
		\left\langle A_1u, v \right\rangle_{L^2(0,1)}=\int_{0}^{1} (a u'')''v\,dx=\int_{0}^{1}au''v''dx=\left\langle u, A_1v \right\rangle_{L^2(0,1)}.
	\end{equation*}
	\indent
	\underline{$A_1$ is non negative}: again by (\ref{gf0 senza bordo}), for any $u\in D(A_1)$, one has
	\begin{equation*}
		\left\langle A_1u, u \right\rangle_{L^2(0,1)}=\int_{0}^{1}(au'')''u\,dx=\int_{0}^{1} a(u'')^2dx\ge 0.
	\end{equation*}
	\indent
	\underline{$I+A_1$ is surjective}: observe that $H^2_{a}(0, 1)$ equipped with the inner product
	\begin{equation*}
		\left\langle u, v \right\rangle_{H^2_{a}(0, 1)} := \int_{0}^{1} \biggl (u v+au''v''\biggr ) dx,\quad \,\,\,\,\,\,\,\forall\; u, v \in H^2_{a}(0, 1),
	\end{equation*} 
	is a Hilbert space. Moreover
	\begin{equation*}
		H^2_{a}(0, 1) \hookrightarrow L^2(0, 1) \hookrightarrow \left(H^2_{a}(0, 1)\right)^*,
	\end{equation*}
	where $\left(H^2_{a}(0, 1)\right)^*$ is the dual space of $H^2_{a}(0, 1)$ with respect to $L^2(0, 1)$. Now, for $f \in L^2(0, 1)$, define the functional $F\in \left(H^2_{a}(0, 1)\right)^*$ given by 
	\begin{equation*}
		F(v):= \int_{0}^{1}f v\,dx,\quad \,\,\,\,\,\,\,\forall \;v\in H^2_{a}(0, 1).
	\end{equation*}
	Consequently, by the Lax-Milgram Theorem, there exists a unique $u\in H^2_{a}(0, 1)$ such that
	\begin{equation}\label{Consequence Riesz's Theorem}
		\left\langle u, v \right\rangle_{H^2_{a}(0, 1)}=\int_{0}^{1} f v\,dx \Longleftrightarrow \int_{0}^{1} \biggl (u v+au''v''\biggr ) dx=\int_{0}^{1} f v\,dx,
	\end{equation}
for all $v\in H^2_{a}(0, 1)$, being $\left\langle \cdot, \cdot \right\rangle_{H^2_{a}(0, 1)}$ the scalar product associated to $|||\cdot |||$.
	In particular, since $\mathcal{C}^{\infty}_c(0, 1) \subset H^2_{a}(0, 1)$, (\ref{Consequence Riesz's Theorem}) becomes
	\begin{equation*}
		\int_{0}^{1} au''v''dx = \int_{0}^{1} (f-u)v\,dx,\quad \,\,\,\,\,\,\,\forall	\;v\;\in \mathcal{C}^{\infty}_c(0, 1).
	\end{equation*}
	Thus, the distributional second derivative of $au''$ is equal to $f-u$ a.e. in $(0,1)$ and it is a function in $L^2(0,1)$, that is $au''\in H^2(0,1)$ \cite[Lemma 2.1]{CF} and $u+A_1u=f$.  This implies that $au''$ and $(au'')'$ are continuous functions in $[0,1]$ and, in particular, $u\in {\mathcal Z}(0,1)$. Moreover, it is possible to prove that 
	\[
	\begin{cases}
	u''(0)=u''(1)=u'''(0)=u'''(1)=0, & x_0 \in (0,1),\\
	u''(1)=u'''(1)=0, (au'')(0)=(au'')'(0)=0, & x_0=0,\\
	u''(0)=u'''(0)=0, (au'')(1)=(au'')'(1)=0, & x_0=1.
	\end{cases}
	\] In fact, by  (\ref{gf0}) and (\ref{Consequence Riesz's Theorem}), one has that for all $u\in {\mathcal Z}(0,1)$ and $v\in H^2_a(0,1)$ the following relation holds:
	\begin{equation*}
		\int_0^1(au'')''v\,dx-[(au'')'v]^{x=1}_{x=0}+[au''v']^{x=1}_{x=0}=\int_{0}^{1} au''v''dx = \int_{0}^{1} (f-u)v\,dx.
	\end{equation*}
	Thus, $[(au'')'v]_{x=0}^{x=1}-[au''v']_{x=0}^{x=1}=0$ for all $v\in H^2_{a}(0, 1)$. Hence, one can conclude that 
	\begin{equation}\label{co}
		(au'')'(1)=(au'')'(0)=0
		\quad \text{  and  } \quad
		(au'')(1)=(au'')(0)=0.
	\end{equation}
	
	Now, just to fix the idea, assume $x_0\in (0,1)$ and set $m:=au''\in H^2(0,1)$. Since $a (x)\neq 0$ for all $x \neq x_0$, we have
	\begin{equation*}
		m(x)=a(x)u''(x) \Longleftrightarrow \frac{m(x)}{a(x)}=u''(x), \quad \forall \; x \in [0,1] \setminus \{x_0\}.
	\end{equation*}
	Hence $u''(0)= \frac{m(0)}{a(0)} =0$ and $u''(1)= \frac{m(1)}{a(1)} =0$. Moreover,
	\begin{equation*}
		(au'')'(x)=a'(x)u''(x)+a(x)u'''(x), \quad \forall \; x \in [0,1] \setminus \{x_0\},
	\end{equation*}
	and, in particular,
	\begin{equation*}
		(au'')'(1)=0 \Longleftrightarrow a'(1)u''(1)+a(1)u'''(1)=0.
	\end{equation*}
	Analogously,
	$
		a'(0)u''(0)+a(0)u'''(0)=0.
$
	However, we know that $u''(1)=0=u''(0)$ and $a(1)\neq 0 \neq a(0)$, thus 
	$
		u'''(1)=u'''(0)=0.
$
	Thus, we can conclude that if $x_0 \in (0,1)$, then
	\begin{equation}\label{su}
		u\in D(A_1)\,\,\,\,\text{ and }\,\,\,\,(I+A_1)(u)=f.
	\end{equation}
	On the other hand, if $x_0=0$, proceeding as before, we have $u''(1)=0=u'''(1)$. This fact, together with \eqref{co}, implies \eqref{su}.
	The same result follows if $x_0=1$.
	
	In every case, as an immediate consequence of the Stone-von Neumann Spectral Theorem and functional calculus associated with the spectral theorem, one has that the operator $(A_1,D(A_1))$ generates a cosine family and an analytic semigroup of angle $\displaystyle\frac{\pi}{2}$ on $L^2(0, 1)$.
\end{proof}

As a consequence of the previous generation property, one has the following well posedness result which completes the discussion of the problem in the weakly degenerate case giving information on the regularity of the solution itself.

First of all, recall the following definition:
\begin{Definition}
	If $u_0 \in L^2(0,1)$ and $h\in L^2(0,T;L^2(0,1))$, a function $u$ is said to be a weak solution of 
	\begin{equation}\label{problema1}
	\begin{cases}
		u_t(t,x)+(au_{xx})_{xx}(t,x)=h(t,x), &(t,x)\in (0,T) \times (0,1),\\
		u(0,x)=u_0(x),&x\in(0,1),
	\end{cases}
\end{equation}
	 with the following boundary conditions
	\[
	(BC)\quad	\begin{cases}
		u_{xx}(t,0)=u_{xxx}(t,0)=u_{xx}(t,1)=u_{xxx}(t,1)=0, &x_0 \in (0,1),\\
		u_{xx}(t,1)=u_{xxx}(t,1)=(au_{xx})(t,0)= (au_{xx})_x(t,0)=0, &x_0=0,\\
		u_{xx}(t,0)=u_{xxx}(t,0)=(au_{xx})(t,1)= (au_{xx})_x(t,1)=0, &x_0=1,
	\end{cases}
	\]
	if
	\begin{equation*}
		u\in \mathcal{C}\Bigl([0,T];L^2(0, 1)\Bigr )\cap L^2\Bigl (0,T;H^2_{a}(0, 1) \Bigr )
	\end{equation*}
	and
	\begin{equation*}
		\begin{split}
			&	\int_0^1u(T,x)\varphi(T,x)\,dx-\int_0^1u_0(x)\varphi(0,x)\,dx-\int_0^T\int_0^1u(t,x)\varphi_t(t,x)\,dx\,dt =\\
			&- \int_0^T\int_0^1a(x)u_{xx}(t,x)\varphi_{xx}(t,x)\,dx\,dt+\int_0^T\int_0^1h(t,x)\varphi(t,x)\,dx\,dt
		\end{split}
	\end{equation*}
	for all $\varphi\in H^1(0,T;L^2(0,1))\cap L^2(0,T;H^2_{a}(0, 1))$.
\end{Definition}
Using the semigroup technique one has the next result.
\begin{Theorem}\label{teo exist and regul caso div}
	For all $h \in L^2(0,T;L^2(0,1))$ and $u_0\in L^2(0,1)$, there exists a unique solution
	\begin{equation*}
		u\in \mathcal{C}\Bigl([0,T];L^2(0, 1)\Bigr )\cap L^2\Bigl (0,T;H^2_{a}(0, 1) \Bigr )
	\end{equation*}
	of \eqref{problema1} satisfying $(BC)$ such that
	\[
		\sup_{t\in [0,T]}\Vert u(t)\Vert^2_{L^2(0, 1)}+\int_0^T\Vert u(t)\Vert^2_{H^2_{a}(0, 1)}dt\le C_T\Bigl (\Vert u_0\Vert^2_{L^2(0, 1)}+\Vert h\Vert^2_{L^2(0,T;L^2(0,1))} \Bigr )
	\]
	for some positive constant $C_T$. Moreover, if $h\in W^{1,1}(0,T;L^2(0,1))$ and $u_0\in H^2_a(0,1)$, then $
		u\in \mathcal{C}^1\Bigl([0,T];L^2(0, 1)\Bigr )\cap \mathcal{C}\Bigl ([0,T];D(A_1) \Bigr).
	$
\end{Theorem}
\subsection{Strongly degenerate setting}
In this subsection we assume that the function $a$ is strongly degenerate. Inspired again by \cite{ACF}, \cite{FGGR} and \cite{fm2013}, we introduce the weighted spaces
\begin{equation}\label{space}
	\begin{aligned}
		H^i_a(0,1):=\{u\in H^{i-1}(0,1):& \; u^{(i-1)} \text{ is locally absolutely continuous in}\\
		& [0,1]\setminus\{x_0\} \text { and } \sqrt{a}u^{(i)}\in L^2(0,1)\},
\end{aligned}\end{equation}
 $i=1,2$, equipped with the norms \eqref{normadiv}. We consider the operator $A_1$ defined in \eqref{ope1} with the same domain as in the weakly degenerate case.
Here ${\mathcal Z}(0,1)$ is as in Section \ref{Section 2}, where in this case the space $H^2_a(0,1)$ is the one defined in \eqref{space}. Thus, since $u\in {\mathcal Z}(0,1)$, $u'$ is locally absolutely continuous in $[0,1]\setminus \{x_0\}$ and not absolutely continuous in $[0,1]$ as for the weakly degenerate case, thus equality (\ref{gf0}) is not true a priori. An idea to prove the Gauss-Green formula given in Lemma \ref{lemma gauss-green} is to characterize $\mathcal Z(0,1)$ and hence $D(A_1)$. 
To do this, define
\begin{equation*}
	\begin{split}
		X:=\{u\in H^1(0,1):\,
		&\text{$u'$ is locally absolutely continuous in $[0,1]\setminus\{x_0\}$},\\
		&au, au' \in H^1(0,1), au''\in H^2(0,1), \sqrt{a}u''\in L^2(0,1),\\
		&(au^{(k)})(x_0)=0,\, \text{for all $k=0,1,2$}
		\}.
	\end{split}
\end{equation*}
Using the definition of $X$ one can easily have the next property.
\begin{Lemma}\label{lemma2.2} For all $u \in X$ we have that
	\begin{enumerate}
		\item $|a(x)u(x)| \le \|(au)'\|_{L^2(0,1)}\sqrt{|x-x_0|}$,
		\item $|a(x)u'(x)| \le \|(au')'\|_{L^2(0,1)}\sqrt{|x-x_0|}$,
		\item $|a(x)u''(x)| \le \|(au'')'\|_{L^2(0,1)}\sqrt{|x-x_0|}$
	\end{enumerate}
	for all $x \in [0,1]$.
\end{Lemma}
\begin{proof} We will prove only the last point, the proof of the other points being similar. Let $u \in X$. Since $(au'')(x_0)=0$, then
	\[
	|(au'')(x)|= \left|\int_{x_0}^x (au'')'(s)ds\right| \le
	\|(au'')'\|_{L^2(0,1)}\sqrt{|x-x_0|},
	\]
	for all $ x \in [0,1]$. 
\end{proof}

Using the assumptions on $a$, in particular the fact that $\ds\frac{1}{a} \not\in L^1(0,1)$, one can prove the next characterization.
\begin{Proposition}\label{Proposition 2.1}
	The spaces $X$ and $\mathcal Z(0,1)$ coincide.
\end{Proposition}
\begin{proof}
	Obviously, $X\subseteq \mathcal Z (0,1)$.
	Now, we prove $\mathcal Z (0,1)\subset X$. To this aim, let $u\in \mathcal Z(0,1)$; we only need to prove that 
	\begin{equation*}
		au\in H^1(0,1),
		\quad
		au'\in H^1(0,1)
	\end{equation*}
	and $(au^{(k)})(x_0)=0,$ for all $k=0,1,2$.
	
	It is easy to prove that $(au)(x_0)=(au'')(x_0)=0.$  Indeed, since $u\in H^1(0,1)$, we have that $u \in \mathcal C[0,1]$, thus $au \in  \mathcal C[0,1]$ and there exists 
	\[
	\lim_{x \rightarrow x_0} (au)(x)=(au)(x_0)=0.
	\]

Now, we prove	 that $(au'')(x_0)=0$. By assumption, we know that $au''\in H^2(0,1)$, hence $au''$  is continuous in $[0,1]$. In particular, there exists $\lim_{x\rightarrow x_0} (au'')(x)= (au'')(x_0)=:N \in \R$. If $N\neq 0$, there exists $C>0$ such that
$
	|(au'')^2(x)|\ge C
$
	for all $x$ in a  neighbourhood of $x_0$, $x \neq x_0$. In particular,
	\[
	|(a(u'')^2)(x)|\ge \frac{C}{a(x)}
	\]
	for all $x$ in a  neighbourhood of $x_0$, $x \neq x_0$. Using the assumption on $a$ one has that $\sqrt{a} u'' \not\in L^2(0,1)$; hence $N=0$.

	Now, we prove that $au \in H^1(0,1)$.
	We already know that $au\in L^2(0,1)$. Furthermore, since $a\in W^{1,\infty}(0,1)$,
	$
		(au)'=a'u+au'\in L^2(0,1).
	$
Moreover,  we easily get that $(au)'$ is the distributional derivative of $au$, and so $au \in H^1(0,1)$. Indeed, assuming for simplicity $x_0 \in (0,1)$, taking $\varphi \in \mathcal C_c^\infty(0,1)$ and $h>0$, one has
	\[
	\begin{aligned}
		\int_0^1(au)'(x)\varphi(x)dx&= \int_0^{x_0-h}(au)'(x)\varphi(x)dx + \int_{x_0-h}^{x_0+h}(au)'(x)\varphi(x)dx\\
		&+ \int_{x_0+h}^1(au)'(x)\varphi(x)dx.
	\end{aligned}\] 
	Obviously, by the absolute continuity of the Lebesgue integral, it follows that 
	\[
	\lim_{h \rightarrow 0}\int_{x_0-h}^{x_0+h}(au)'(x)\varphi(x)dx=0;
	\]
	moreover
	\[
	\begin{aligned}
		&\lim_{h \rightarrow 0}\int_0^{x_0-h}(au)'(x)\varphi(x)dx =\lim_{h \rightarrow 0}\left(( au\varphi)(x_0-h) - \int_0^{x_0-h}(au)(x)\varphi '(x)dx\right) \\
		&= - \int_0^{x_0}(au)(x)\varphi '(x)dx 
	\end{aligned}
	\]
	and 
	\[
	\begin{aligned}
		&
		\lim_{h \rightarrow 0}\int_{x_0+h}^1(au)'(x)\varphi(x)dx = \lim_{h \rightarrow 0}\left(- (au\varphi)(x_0+h) - \int_{x_0+h}^1(au)(x)\varphi '(x)dx\right) 
		\\
		&=  - \int_{x_0}^1(au)(x)\varphi '(x)dx.
	\end{aligned}
	\]
	Hence, for all  $\varphi \in \mathcal C_c^\infty(0,1)$ it holds
	\[
	\int_0^1(au)'(x)\varphi(x)dx = -\int_0^1(au)(x)\varphi '(x)dx.
	\]
Nothing change if $x_0=0$ or $x_0=1$;	thus $(au)'$ is the distributional derivative of $au$ and so $au \in H^1(0,1)$. 

Now, we prove that $\exists \; \lim_{x \rightarrow 0} (au')(x)=(au')(x_0)=0$.
	As a first step, consider $x_0 \in (0,1)$ and take $x<x_0$. Since $u'$ is locally absolutely continuous in $[0,1] \setminus \{x_0\}$, we have that $u'(0)$ and $u'(1)$ are well defined. Moreover,  $au' \in L^2(0,1)$ and
	\[
	(au')' =a'u' + au'' \in L^2(0,1),
	\]
	since $u \in \mathcal Z(0,1)$.
	 Thus
	\begin{equation*}
		(au')(x)=\int_0^x (au')'(t)dt+(au')(0)
	\end{equation*}
	and there exists $\lim_{x \to x_0^-}(au')(x)=(au')(x_0^-)=\int_{0}^{x_0}(au')'(t)dt+(au')(0) =:L \in \R$. Analogously, for $x>x_0$: starting from
	\begin{equation*}
		(au')(x)=-\int_x^1 (au')'(t)dt+(au')(1),
	\end{equation*}
	one can prove that $\exists\;\lim_{x \to x_0^+}(au')(x)=(au')(x_0^+)=-\int_{x_0}^{1}(au')'(t)dt+(au')(1)=:M \in \R$. Proceeding as before, we can prove that $L=M=0$. 
	Indeed, if we assume $L, M\neq 0$, then there exists $C>0$ such that
	\begin{equation*}
		|(u')^2(x)|\ge  \frac{C}{a(x)},
	\end{equation*}
	for all $x$ in a  left and in a right  neighbourhood of $x_0$. Using the assumption on $a$, one has that $u' \notin L^2(0,1)$; hence $L=M=0$. 
	Thus, there exists $\lim_{x \rightarrow x_0 } (au')(x)= (au')(x_0)=0$.
	
	If $x_0=0$ or $x_0=1$, one can proceed as in the case $x_0 \in (0,1)$, taking $x>0$ or $x<1$, respectively. 
	
	Proceeding as before, one can prove that  $(au')'$ is the distributional derivative of $au'$  so that $au' \in H^1(0,1)$ and the thesis follows.
\end{proof}
Clearly, the following characterizations hold:
\[
D(A_1):=\{u\in X:\, u''(0)=u''(1)=u'''(0)=u'''(1)=0
\}, \quad \text{if} \;\; x_0 \in (0,1),
\]
\[
	\begin{split}
		D(A_1):=\{u\in X:\; u''(1)=u'''(1)=0, (au'')(0)=(au'')'(0)=0\},\quad \text{if} \;\; x_0 =0,
	\end{split}
\]
\[
	\begin{split}
		D(A_1):=\{u\in X: u''(0)=u'''(0)=0, (au'')(1)=(au'')'(1)=0\},\quad \text{if} \;\; x_0 =1.
	\end{split}
\]

Using Proposition \ref{Proposition 2.1} and Lemma \ref{lemma2.2}, we can prove the formula of integration by parts (\ref{gf0}) also in the strongly degenerate case.
\begin{Lemma}\label{gf0 strong case}
	For all $(u,v)\in {\mathcal Z}(0,1) \times H^2_a(0,1)$ the Gauss-Green formula \eqref{gf0} still holds. 
\end{Lemma}
\begin{proof}
	Take $u \in \mathcal Z(0,1)$ and $v \in H^2_a(0,1)$. By definitions of $\mathcal Z(0,1)$ and $H^2_a(0,1)$, one has that $au''\in H^2(0,1)$ and $v \in H^1(0,1)$. Hence, one can integrate by parts, obtaining
	\begin{equation}\label{sg1}
		\int_{0}^{1}(au'')''v\,dx=[(au'')'v]^{x=1}_{x=0}-\int_{0}^{1}(au'')'v'dx.
	\end{equation}
	Now, we consider the term $\int_{0}^{1}(au'')'v'dx$ and take $x_0 \in (0,1)$.  For $\delta >0$, we have
	\begin{equation}\label{integrationS}
		\begin{aligned} \int_0^1 (au'')' v' dx &= (au''v')(x_0- \delta) - (au''v')(0)
			-\int_0^{x_0-\delta}au''v''dx\\& + \int_{x_0-
				\delta}^{x_0+\delta}(au'')'v' dx + (au''v')(1) - (au''v')(x_0+ \delta) \\
				&- \int_{x_0+
				\delta}^1 au''v'' dx.
		\end{aligned}
	\end{equation}
	Moreover, by the  H\"older inequality
	$a u''v''$ and $(au'')' v'$ belong to $L^1(0, 1)$ (observe that $au''v''=\sqrt{a}u'' \sqrt{a}v'' \in L^1(0,1)$). Hence, by the absolute continuity of the Lebesgue integral, one has
	\begin{equation}\label{delta1S}
		\lim_{\delta \rightarrow 0} \int_0^{x_0- \delta}au''v''  dx =
		\int_0^{x_0}au''v''  dx, \quad \lim_{\delta \rightarrow 0} \int_{x_0+
			\delta}^1au''v''  dx = \int_{x_0}^1au''v'' dx
	\end{equation}
	and
	\begin{equation}\label{delta2S}
		\lim_{\delta \rightarrow 0} \int_{x_0- \delta}^{x_0+ \delta}(au'')'v'
		dx =0.
	\end{equation}
	In order to obtain the
	desired result it is sufficient to prove that
	\begin{equation}\label{delta3S}
		\lim_{\delta \rightarrow 0} (au''v')(x_0- \delta) =\lim_{\delta
			\rightarrow 0} (au''v')(x_0+ \delta).
	\end{equation}
	
	To this aim, we can choose $\delta$ in \eqref{integrationS} such that $v'$ is absolutely continuous in $\mathcal K:=[0, x_0-\delta]\cup [x_0+ \delta, 1]$. Moreover, using the fact that also $au''\in \mathcal C[0,1]$, one has 
	\[
	\begin{aligned}
		(au''v')(x_0-\delta)&= \int_0^{x_0-\delta}(au''v')'(s)ds +(au''v')(0)\\
		&=\int_0^{x_0-\delta}[(au'')'v'](s)ds +\int_0^{x_0-\delta}[au''v''](s)ds +(au''v')(0).
	\end{aligned}
	\]
	Proceeding as before, one has that there exists 
	$L_1\in \R$ such that
	\[
	\begin{aligned}
		\lim_{\delta \rightarrow 0}(au''v')(x_0-\delta)
		&=\lim_{\delta \rightarrow 0}\left( \int_0^{x_0-\delta}[(au'')'v'](s)ds +\int_0^{x_0-\delta}[au''v''](s)ds \right)\\
		&+(au''v')(0)\\
		& =\int_0^{x_0}[(au'')'v'](s)ds + \int_0^{x_0}(au''v'')(s)ds +(au''v')(0) =:L_1.
	\end{aligned}
	\]
	Analogously, one has
	\[
	\begin{aligned}
		\lim_{\delta \rightarrow 0}(au''v')(x_0+\delta)&=(au''v')(1)-\int_{x_0}^1[(au'')'v'](s)ds-\int_{x_0}^1(au''v'')(s)ds =:L_2.
	\end{aligned}
	\]
	If $L_1 \neq 0$, then there exists $C>0$ such that
$
	|(au''v')(x)| \ge C
$
	for all $x$ in a left  neighbourhood of $x_0$, $x \neq x_0$. Thus, by
	Lemma \ref{lemma2.2},
	\[
	|v'(x)| \ge \frac{C}{|(au'')(x)|}\ge \frac{C_1}{\sqrt{x_0-x}}
	\]
	for all $x$ in a left  neighbourhood of $x_0$, $x \neq x_0$, and for a
	suitable positive constant $C_1$. This implies that $v' \not \in
	L^2(0,1)$. Hence $L_1=0$. Analogously, one can prove that $L_2=0$.
	Thus \eqref{delta3S} holds and 
	the thesis follows by \eqref{sg1}-\eqref{delta3S}.
	
	The case $x_0=0$ or $x_0=1$ can be proved in a similar way.
\end{proof}
Also in the strongly degenerate case, to prove Theorem \ref{teorema generazione (div)}    it is important the equivalence between the two norms $\|\cdot\|_{H^2_a(0,1)}$ and $|||\cdot|||$ (defined in \eqref{norma tre stanghette}). This is done in \cite[Proposition 2.2]{CF-articolo Wentzell} under the additional Hypothesis:  
\begin{Assumptions}\label{Ipoaggiuntiva articolo wentzell}
	Assume that there exists $K \in [1,2)$ such that the function 
	$
	x \mapsto \frac{|x-x_0|^K}{a}
	$
	is
	\begin{enumerate}
		\item non increasing on the left of $x_0$ and non decreasing on the right of $x_0$, if $x_0 \in (0,1),$
		\item non decreasing on the right of $0$, if $x_0=0,$
		\item non increasing on the left of $1$, if $x_0=1$.
	\end{enumerate}
\end{Assumptions}

Thus, the analogues of Theorems \ref{teorema generazione (div)} and \ref{teo exist and regul caso div} hold under Hypothesis \ref{Ipoaggiuntiva articolo wentzell} if $a$ is strongly degenerate.

\section{The operator in non divergence form}\label{Sec3}
In this section we consider the operator $A_2u:= au_{xxxx}$
with suitable boundary conditions and we prove its generation property.

In this case the spaces that we consider are the same in the weakly or in the strongly degenerate case. In particular, we take into account the following spaces introduced in \cite{CaFrRo}, if the degeneracy occurs on the boundary of the space domain, or in \cite{6} (see also \cite{memoirs}), if the degeneracy point is in the interior:
\begin{equation*} 
	L^2_{\frac{1}{a}}(0, 1):=\biggl \{u\in L^2(0, 1):\int_{0}^{1}\frac{u^2}{a}\,dx<+\infty \biggr \}
\end{equation*}
and
\begin{equation*}
	H^i_{\frac{1}{a}}(0, 1):= L^2_{\frac{1}{a}}(0, 1)\cap H^i(0, 1),
\end{equation*}
with the respective norms defined by
\begin{equation*}
	\norm{u}^2_{L^2_{\frac{1}{a}}(0, 1)}:= \int_{0}^{1}\frac{u^2}{a}\,dx, \quad\,\,\,\,\,\,\forall \;u \in L^2_{\frac{1}{a}}(0, 1),
\end{equation*}
and
\begin{equation*}
	\norm{u}_{H^i_{\frac{1}{a}}(0, 1)}^2:=\norm{u}^2_{L^2_{\frac{1}{a}}(0, 1)} + \sum_{j=1}^{i}\| u^{(j)}\|^2_{L^2(0, 1)}, \quad\,\,\,\, \forall \;u \in H^i_{\frac{1}{a}}(0, 1),
\end{equation*}
$i=1,2$.
Observe that for all $u \in H^i_{\frac{1}{a}}(0, 1)$, one can prove that $\|u\|^2_{H^i_{\frac{1}{a}}(0,1)}$ is equivalent to the following one
\[
\|u\|_i^2:= \|u\|^2_{L^2_{\frac{1}{a}}(0,1)} +  \|u^{(i)}\|^2_{L^2(0,1)}, \quad \forall \; u \in H^i_{\frac{1}{a}}(0, 1)
\]
(see, e.g., \cite[Chapter VIII]{brezis} or \cite{CF-articolo Wentzell}).  We underline also that if $u \in H^2_{\frac{1}{a}}(0, 1)$, then $(au)(x_0)=(au')(x_0)=0$ since $a(x_0)=0$.

Also in this section we distinguish between the weakly degenerate case and the strongly degenerate one.

\subsection{Weakly degenerate case}
Assume that the function $a$ is weakly degenerate at a point $x_0 \in [0,1]$ and define the space
\begin{equation}\label{spazioW}
	\mathcal{W}(0,1):=\Bigl \{u\in H^2_{\frac{1}{a}}(0, 1):au''''\in L^2_{\frac{1}{a}}(0, 1)\Bigr \}.
\end{equation}
Clearly, if $u \in \mathcal W(0,1)$, then $u'''' \in L^1(0,1)$; indeed
\[
\int_0^1 u''''(x)dx = \int_0^1 \frac{\sqrt{a}u''''(x)}{\sqrt{a}}dx \le \left\|\frac{1}{a}\right\|_{L^1(0,1)} \|\sqrt{a}u''''\|_{L^2(0,1)}.
\] Thus, we can apply \cite[Lemma 2.1]{CF} obtaining that $u \in W^{4,1}(0,1)$  and, as a consequence, $u \in \mathcal C^3[0,1]$.

In particular,  being $a(x_0)=0$, $(au'')(x_0)=(au''')(x_0)=0$. 

Moreover, as in \cite[Proposition 3.1]{CF} for the Dirichlet case, one can prove 
that the spaces $H^i_{\frac{1}{a}}(0, 1)$ 
and $H^i(0, 1)$, $i=1,2$, coincide algebraically and the two norms are equivalent. 

Hence, as in Lemma \ref{lemma gauss-green}, one has immediately the following integration by parts.
\begin{Lemma}\label{lemma problema di prova}
	For all $(u,v)\in \mathcal{W}(0,1)\times H^2_{\frac{1}{a}}(0, 1)$ one has
	\begin{equation}\label{GF1}
		\int_{0}^{1}u''''v\,dx=[u'''v]^{x=1}_{x=0}-[u''v']^{x=1}_{x=0}+\int_{0}^{1}u''v''dx.
	\end{equation}
\end{Lemma}
Using the previous spaces, we define the operator $A_2$ by
\[
A_2u:=au''''
\] 
for all $u \in
D(A_2)$, where
\begin{equation*}
	\begin{aligned}
		D(A_2):&=\{u\in \mathcal{W}(0,1):\, u''(0)=u''(1)= u'''(0)=u'''(1)=0 \}\\
		&=\left\{ u\in H^2_{\frac{1}{a}}(0,1): u''(0)=u''(1)= u'''(0)=u'''(1)=0\right. \\
		 &\left. \qquad \qquad \qquad \quad
		 \text{ and }au''''\in L^2_{\frac{1}{a}}(0, 1) \right\}.
	\end{aligned}
\end{equation*}
In this case the definition of $D(A_2)$ and, in particular, the natural Neumann boundary conditions
\begin{equation*}
	u''(0)=u''(1)=u'''(0)=u'''(1)=0
\end{equation*}
are independent of the degeneracy point $x_0\in [0,1]$ (recall that $u \in \mathcal C^3[0,1]$, thus all the previous terms are well defined).

\begin{Remark}
	As observed in Remark \ref{remark 1}, we underline that, if $u\in D(A_2)$ and $v\in H^2_{\frac{1}{a}}(0, 1)$, then the boundary terms in \eqref{GF1}, i.e., $[u'''v]^{x=1}_{x=0}$ and $[u''v']^{x=1}_{x=0}$, are equal to $0$ and (\ref{GF1}) becomes
	\begin{equation}\label{rem2}
		\int_{0}^{1}u''''v\,dx=\int_{0}^{1}u''v''dx.
	\end{equation}
\end{Remark}
As in the previous section, this formula will represent the key ingredient in the generation property established in the next theorem.
\begin{Theorem}\label{Theorem 3.1}
	The operator $(A_2,D(A_2))$ is self-adjoint and non negative on $L^2_{\frac{1}{a}}(0, 1)$ with dense domain. Hence $-A_2$ generates a contractive analytic semigroup of angle $\displaystyle\frac{\pi}{2}$ on $L^2_{\frac{1}{a}}(0, 1)$. 
\end{Theorem}
\begin{proof}
	The proof is similar to the one of Theorem \ref{teorema generazione (div)}, so we sketch it.
	
	\underline{$A_2$ is symmetric}: using \eqref{rem2}, for any $u,v\in D(A_2)$, one has
	\begin{equation*}
		\left\langle v, A_2u \right\rangle_{L^2_{\frac{1}{a}}(0, 1)}=\int_{0}^{1}\frac{v a u''''}{a}\,dx=\int_{0}^{1}v''''u\,dx=\left\langle A_2v, u \right\rangle_{L^2_{\frac{1}{a}}(0, 1)}.
	\end{equation*}
	\indent
	\underline{$A_2$ is non negative}: for any $u\in D(A_2)$, by \eqref{rem2},
	\begin{equation*}
		\left\langle A_2u, u \right\rangle_{L^2_{\frac{1}{a}}(0, 1)}=\int_{0}^{1}\frac{a u'''' u}{a}\,dx=\int_{0}^{1}(u'')^2dx\ge 0.
	\end{equation*}
	\indent
	\underline{$I+A_2$ is surjective}:  observe that $H^2_{\frac{1}{a}}(0, 1)$ equipped with the inner product
	\begin{equation*}
		\left\langle u, v \right\rangle_{H^2_{\frac{1}{a}}(0, 1)} := \int_{0}^{1} \biggl (\frac{u v}{a}+u''v''\biggr ) dx\quad \,\,\,\,\,\,\,\forall\; u, v \in H^2_{\frac{1}{a}}(0, 1),
	\end{equation*} 
	is a Hilbert space. Moreover
	\begin{equation*}
		H^2_{\frac{1}{a}}(0, 1) \hookrightarrow L^2_{\frac{1}{a}}(0, 1) \hookrightarrow \left(H^2_{\frac{1}{a}}(0, 1)\right)^*,
	\end{equation*}
	where $\left(H^2_{\frac{1}{a}}(0, 1)\right)^*$ is the dual space of $H^2_{\frac{1}{a}}(0, 1)$ with respect to $L^2_{\frac{1}{a}}(0, 1)$. 
	Now, for $f \in L^2_{\frac{1}{a}}(0, 1)$, define the functional $F\in \left(H^2_{\frac{1}{a}}(0, 1)\right)^*$ given by 
	\begin{equation*}
		F(v):= \int_{0}^{1}\frac{f v}{a}\,dx\quad \,\,\,\,\,\,\,\forall \;v\in H^2_{\frac{1}{a}}(0, 1).
	\end{equation*}
	Consequently, by the Lax-Milgram Theorem, there exists a unique $u\in H^2_{\frac{1}{a}}(0, 1)$ such that
	\begin{equation}\label{Consequence Riesz}
		\left\langle u, v \right\rangle_{H^2_{\frac{1}{a}}(0, 1)}=\int_{0}^{1} \frac{f v}{a}\,dx \Longleftrightarrow \int_{0}^{1} \biggl (\frac{u v}{a}+u''v''\biggr ) dx=\int_{0}^{1} \frac{f v}{a}\,dx,
	\end{equation}
for all $v\in H^2_{\frac{1}{a}}(0, 1)$.	In particular, since $\mathcal{C}^{\infty}_c(0, 1) \subset H^2_{\frac{1}{a}}(0, 1)$, (\ref{Consequence Riesz}) becomes
	\begin{equation*}
		\int_{0}^{1} u''v''dx = \int_{0}^{1} \frac{f-u}{a}\,v\,dx\quad \,\,\,\,\,\,\,\forall	\;v\;\in \mathcal{C}^{\infty}_c(0, 1).
	\end{equation*}
	Thus, the distributional second derivative of $u''$ is equal to $\ds\frac{f-u}{a}$ a.e. in $(0,1)$ and, proceeding as in \cite{CF}, one has 
	$au''''\in L^2_{\frac{1}{a}}(0, 1)$; hence $u \in \mathcal W(0,1)$. Coming back to (\ref{Consequence Riesz}) and using the Gauss-Green Identity, it results that
	\begin{equation*}
		[u'''v]^{x=1}_{x=0}=[u''v']^{x=1}_{x=0}=0\quad\,\,\,\,\,\,\,\,\,\,\forall\; v \in H^2_{\frac{1}{a}}(0, 1)
	\end{equation*}
	and one can conclude that $u''(0)=u''(1)=u'''(0)=u'''(1)=0$.
	
	As a consequence, $u \in D(A_2)$ and
	$
		au''''=f -u\quad \text{a.e. in } (0,1).
$
	Hence $(I+A_2)(u)=f$. The rest of the proof is as in Theorem \ref{teorema generazione (div)}.
\end{proof}

Hence, one has the following well posedness result, for which we premise the following definition:
\begin{Definition}
	If $u_0 \in L^2_{\frac{1}{a}}(0, 1)$ and $h\in L^2(0,T;L^2_{\frac{1}{a}}(0, 1))$, a function $u$ is said to be a weak solution of 
	\begin{equation}\label{problema2}
	\begin{cases}
		u_t(t,x)+au_{xxxx}(t,x)=h(t,x), &(t,x)\in (0,T) \times (0,1),\\
		u_{xx}(t,0)=u_{xx}(t,1)=0,&t\in (0,T),\\
		u_{xxx}(t,0)=u_{xxx}(t,1)=0,&t\in (0,T),\\
		u(0,x)=u_0(x),&x\in(0,1),
	\end{cases}
\end{equation}
	if
	\begin{equation*}
		u\in \mathcal{C}\Bigl([0,T];L^2_{\frac{1}{a}}(0, 1)\Bigr )\cap L^2\Bigl (0,T;H^2_{\frac{1}{a}}(0, 1) \Bigr )
	\end{equation*}
	and
	\begin{equation*}
		\begin{split}
			&	\int_0^1\frac{u(T,x)\varphi(T,x)}{a(x)}\,dx-\int_0^1\frac{u_0(x)\varphi(0,x)}{a(x)}\,dx-\int_0^T\int_0^1\frac{u(t,x)\varphi_t(t,x)}{a(x)}\,dx\,dt =\\
			&- \int_0^T\int_0^1u_{xx}(t,x)\varphi_{xx}(t,x)\,dx\,dt+\int_0^T\int_0^1h(t,x)\frac{\varphi(t,x)}{a(x)}\,dx\,dt
		\end{split}
	\end{equation*}
	for all $\varphi\in H^1(0,T;L^2_{\frac{1}{a}}(0, 1))\cap L^2(0,T;H^2_{\frac{1}{a}}(0, 1))$.
\end{Definition}

Therefore, using the semigroup theory, one can prove the following well posedness theorem.
\begin{Theorem}\label{Theorem 3.2}
	For all $h\in L^2(0,T;L^2_{\frac{1}{a}}(0,1))$ and $u_0\in L^2_{\frac{1}{a}}(0, 1)$, there exists a unique solution
	\begin{equation*}
		u\in \mathcal{C}\Bigl([0,T];L^2_{\frac{1}{a}}(0, 1)\Bigr )\cap L^2\Bigl (0,T;H^2_{\frac{1}{a}}(0, 1) \Bigr )
	\end{equation*}
	of (\ref{problema2}) such that
	\[
		\sup_{t\in [0,T]}\Vert u(t)\Vert^2_{L^2_{\frac{1}{a}}(0, 1)}+\int_0^T\Vert u(t)\Vert^2_{H^2_{\frac{1}{a}}(0, 1)}dt\le C_T\left (\Vert u_0\Vert^2_{L^2_{\frac{1}{a}}(0, 1)}+\Vert h\Vert^2_{L^2(0,T;L^2_{\frac{1}{a}}(0,1))} \right)
	\]
	for some positive constant $C_T$. In addition, if $h\in W^{1,1}(0,T;L^2_{\frac{1}{a}}(0,1))$ and $u_0\in H^2_{\frac{1}{a}}(0,1)$, then
	$
		u\in \mathcal{C}^1\Bigl([0,T];L^2_{\frac{1}{a}}(0, 1)\Bigr )\cap \mathcal{C}\Bigl ([0,T];D(A_2) \Bigr ).
$
\end{Theorem}

\subsection{Strongly degenerate case}
In this subsection we will assume that the function $a$ is strongly degenerate. As in the previous subsections, we need to prove a formula similar to \eqref{GF1}. To this aim we need an additional assumption on the function $a$, which is not surprising because already used in other papers and it is satisfied by the prototype function we have in mind.

\begin{Assumptions}\label{hp 3.3}
	Assume that
	there exist $K \in [1,2)$ and $C>0$ such that 
	\[
	\frac{1}{a(x)}\le \frac{C}{|x-x_0|^K}
	\]
	for all $x \in [0,1]\setminus\{x_0\}$.
\end{Assumptions}
Observe that the previous hypothesis is obviously satisfied by $a(x):=|x-x_0|^K$, where $K\in [1,2)$, and it is more general then the one made in the divergence case, Hypothesis \ref{Ipoaggiuntiva articolo wentzell}.

\vspace{0.5cm}
Now, we introduce
\begin{equation*}
	X:=\biggl \{u\in H^2_{\frac{1}{a}}(0, 1): u(x_0)=0 \biggr \}.
\end{equation*}

Proceeding as in \cite{CaFrRo} and \cite{FGGR}, one can prove the following result. 
\begin{Proposition} \label{Prop 3.2}
	If Hypothesis \ref{hp 3.3} is satisfied, then
	\begin{equation*}
		H^2_{\frac{1}{a}}(0, 1)=X.
	\end{equation*}
\end{Proposition}
We underline that this characterization holds if $x_0 \in [0,1]$.

Now, consider the space $\mathcal W(0,1)$ defined in \eqref{spazioW}.
Clearly, if Hypothesis \ref{hp 3.3} is satisfied, we can rewrite $\mathcal W(0,1)$ as \[
\mathcal{W}(0,1)=\Bigl \{u\in H^2_{\frac{1}{a}}(0, 1):u(x_0)=0 \text{ and } au''''\in L^2_{\frac{1}{a}}(0, 1) \Bigr \}.
\]

The next proposition holds.
\begin{Proposition}\label{regolarità appendice}Assume that the function $a$ is strongly degenerate and Hypothesis \ref{hp 3.3} holds.

	 $\star$ If $x_0\in (0,1)$, then for all $y \in \mathcal W(0,1)$, $y'' \in W^{1,1}(0, x_0)$ and $y'' \in W^{1,1}(x_0,1)$.

$\star$ If $x_0=0$ or $x_0=1$, then for all $y \in \mathcal W(0,1)$, $y'' \in W^{1,1}(0,1)$.
\end{Proposition}
\begin{proof} Assume $x_0=0$, being the other cases similar. Fix $y \in \mathcal W(0,1)$ and take $\delta >0$. Clearly
	\begin{equation}\label{y secondo in delta*}
			y''(\delta)=y''(1)-\int_\delta^1y'''(x)dx,
			\end{equation}
			thus
			\begin{equation}\label{y secondo in delta}
		\begin{aligned}
			y''(\delta)&=y''(1)-\int_\delta^1\Biggl (y'''(1)-\int_x^1y''''(s)ds \Biggr )dx\\
			&	=y''(1)-y'''(1)+\delta y'''(1)+\int_\delta^1y''''(s)(s-\delta)ds.
		\end{aligned}
	\end{equation}
Trivially $\displaystyle\lim_{\delta\to 0}\delta y'''(1)= 0$; moreover
\[
	\int_\delta^1y''''(s)(s-\delta)ds=\int_\delta^1sy''''(s)ds-\delta \int_\delta^1y''''(s)ds.
\]
Since $sy''''(s)\in L^1(0,1)$, $\displaystyle\lim_{\delta\to 0}\int_\delta^1 sy''''(s)ds = \int_0^1 sy''''(s)ds$, by the absolute continuity of the integral,
and\begin{equation*}
	\begin{aligned}
		0&<\delta \int_\delta^1|y''''(s)|ds\le \delta \Biggl (\int_\delta^1\frac{1}{a(s)}ds\Biggr )^{\frac{1}{2}}\Vert\sqrt{a}y''''\Vert_{L^2(0,1)}\\
		&\le\delta^{1- \frac{K}{2}} \Biggl (\int_\delta^1\frac{s^K}{a(s)}ds \Biggr )^{\frac{1}{2}}\Vert\sqrt{a}y''''\Vert_{L^2(0,1)}\le C\delta^{1- \frac{K}{2}}(1-\delta)^{\frac{1}{2}}\Vert\sqrt{a}y''''\Vert_{L^2(0,1)},	\end{aligned}
\end{equation*} for a positive constant $C$.
If we pass to the limit as $\delta\to 0$ in (\ref{y secondo in delta}), we conclude that
\begin{equation*}
	\exists \lim_{\delta\to 0}y''(\delta)= y''(1)- y'''(1)+\int_0^1sy''''(s)ds \in\mathbb{R}.
\end{equation*}
By continuity, it is possible to define
$
	y''(0):=\lim_{\delta\to 0}y''(\delta).
$
In particular, by \eqref{y secondo in delta*},
\begin{equation*}
	\int_0^1y'''(x)dx=y''(1)-y''(0)
\end{equation*}
and the thesis follows.
\end{proof}

As a consequence, one can prove the next Gauss-Green formula. 
\begin{Lemma}\label{green} Assume Hypothesis \ref{hp 3.3}.

	 $\star$ If $x_0\in (0,1)$, then for all $(u,v)\in \mathcal W(0,1)\times H^2_{\frac{1}{a}}(0, 1)$, equality \[\int_0^1 u''''vdx= [u'''v]^{x=1}_{x=0}-[u''v']^{x=1}_{x=0}+ [u''v']^{x_0^+}_{x_0^-} +\int_0^1 u''v'' dx\]
holds. Here $u''(x_0^+)= \lim_{\delta\rightarrow 0^+} u''(x_0+\delta)$, $u''(x_0^-)= \lim_{\delta\rightarrow 0^+} u''(x_0-\delta)$ and $v'(x_0^+)= v'(x_0^-)= v'(x_0)$.

$\star$ If $x_0=0$, then for all $(u,v)\in  \mathcal W(0,1)\times H^2_{\frac{1}{a}}(0, 1)$
		\begin{equation*}
		\int_{0}^{1}u''''v\,dx=u'''(1)v(1)-[u''v']^{x=1}_{x=0}+\int_{0}^{1}u''v''dx.
	\end{equation*}

$\star$	If $x_0=1$, then for all $(u,v)\in  \mathcal W(0,1)\times H^2_{\frac{1}{a}}(0, 1)$
\begin{equation*}
	\int_{0}^{1}u''''v\,dx=-u'''(0)v(0)-[u''v']^{x=1}_{x=0}+\int_{0}^{1}u''v''dx.
\end{equation*}

\end{Lemma}

\begin{proof}
	
Just to fix the idea, we will prove the thesis if  $x_0 \in (0,1)$, being the other proofs similar to this case. Take $\delta >0$ and $\mathcal K:= [0, x_0-\delta ] \cup [x_0+ \delta, 1]$.
	
	By definition of $\mathcal W(0,1)$, $u''''\in L^2(\mathcal K)$, thus, by \cite[Lemma 2.1]{CF}, $u \in H^4(\mathcal K)$. Hence, we can integrate by parts, obtaining
	
	\begin{equation*}
		\begin{split}
			&\int_{0}^{1}u''''v\,dx
			=\int_{0}^{x_0-\delta}u''''v\,dx+\int_{x_0-\delta}^{x_0+\delta}u''''v\,dx+\int_{x_0+\delta}^{1}u''''v\,dx \\
			&=\int_{0}^{x_0-\delta}u''v''dx+\int_{x_0-\delta}^{x_0+\delta}u''''v\,dx+\int_{x_0+\delta}^{1}u''v''dx\\
			&-  (u'''v)(0)+(u''v')(0)+  (u'''v)(1)-(u''v')(1)\\
			& + (u'''v)(x_0-\delta)-  (u'''v)(x_0+\delta)- (u''v')(x_0-\delta)+ (u''v')(x_0+\delta).
		\end{split}
	\end{equation*}
	As in \cite[Proposition 2.1]{CF}, one can prove
	\begin{equation*}
		\lim_{\delta\to 0}\int_{0}^{x_0-\delta}u''v''dx=\int_{0}^{x_0}u''v''dx,\,\,\,\,\,\lim_{\delta\to 0}\int_{x_0+\delta}^{1}u''v''dx=\int_{x_0}^{1}u''v''dx
	\end{equation*}
	and
	$
		\lim_{\delta\to 0}\int_{x_0-\delta}^{x_0+\delta}u''''v\,dx=0.
$
By Proposition \ref{regolarità appendice}, $\lim_{\delta \rightarrow0} (u''v')(x_0-\delta)= u''(x_0^-)v'(x_0)$ and $\lim_{\delta \rightarrow0} (u''v')(x_0+\delta)= u''(x_0^+)v'(x_0)$. \\
Moreover,
$\lim_{\delta \rightarrow0} (u'''v)(x_0-\delta)$ $= \lim_{\delta \rightarrow0} (u'''v)(x_0+\delta)=0$. Indeed
\[
\begin{aligned}
|(u'''v)(x_0+\delta)|&\le |v(x_0+\delta) u'''(1)| + \left|v(x_0+ \delta)\int_{x_0+\delta}^1 u''''(s)ds\right|\\
& \le  |v(x_0+\delta) u'''(1)| + C|v(x_0+ \delta)|	\left(\int_{x_0+\delta}^1 (s-x_0)^{-k}ds\right)^{\frac{1}{2}}\\
&\le \begin{cases}
|v(x_0+\delta) u'''(1)| + C \frac{\delta^{\frac{3-K}{2}}}{\sqrt{K-1}} \frac{|v(x_0+ \delta)|}{\delta}, & K \neq 1\\
|v(x_0+\delta) u'''(1)| + C (-\delta^2\log \delta)^{\frac{1}{2}}\frac{|v(x_0+ \delta)|}{\delta},& K=1.
\end{cases}
\end{aligned}
\]
In any case $ \lim_{\delta \rightarrow0} (u'''v)(x_0+\delta)=0$, since $v(x_0)=0$. Analogously, $ \lim_{\delta \rightarrow0} (u'''v)(x_0-\delta)=0$.
\end{proof}

Now, define the operator $A_2$ by 
\[
A_2u:=au''''
\]
for all $u \in D(A_2)$, where
\[
\begin{aligned}
	D(A_2):&=\{ u\in \mathcal W(0,1): u''(0)=u''(1)=u'''(0)=u'''(1)=0\},
\end{aligned}
\]
if $x_0 \in (0,1)$,
\[
\begin{aligned}
	D(A_2):&=\{ u\in \mathcal W(0,1):u''(0)=0=u''(1)=u'''(1) \},
\end{aligned}
\]
if $x_0=0$ and
\[
\begin{aligned}
	D(A_2):&=\{ u\in \mathcal W(0,1):u''(1)=0=u''(0)=u'''(0) \},
\end{aligned}
\]
if $x_0=1$.

Observe that if $x_0=0$ and $u \in \mathcal W(0,1)$, then $u''(0)$ is well defined thanks to Proposition \ref{regolarità appendice}. In this case, the other boundary terms in $D(A_2)$ are well defined since  far  from the degeneracy point the function $u \in \mathcal W(0,1)$ is $\mathcal C^3$. Analogously if $x_0=1$ or $x_0 \in (0,1)$.

Thanks to Proposition \ref{Prop 3.2}, the following characterizations of $D(A_2)$ hold under Hypothesis \ref{hp 3.3}:
\[
\begin{aligned}
	D(A_2):&=\{ u\in H^2_{\frac{1}{a}}(0, 1):u(x_0)=0= u''(0)=u''(1)=u'''(0)=u'''(1) \\
	& \qquad  \qquad\qquad\quad \;\; \text{ and } au''''\in L^2_{\frac{1}{a}}(0, 1) \},
\end{aligned}
\]
if $x_0 \in (0,1)$,
\[
\begin{aligned}
	D(A_2):&=\{ u\in H^2_{\frac{1}{a}}(0, 1):u(0)=u''(0)=0=u''(1)=u'''(1) \\
	& \qquad  \qquad\qquad\quad \;\; \text{ and } au''''\in L^2_{\frac{1}{a}}(0, 1) \},
\end{aligned}
\]
if $x_0=0$ and
\[
\begin{aligned}
	D(A_2):&=\{ u\in H^2_{\frac{1}{a}}(0, 1):u(1)=u''(1)=0=u''(0)=u'''(0) \\
	& \qquad  \qquad\qquad\quad \;\; \text{ and } au''''\in L^2_{\frac{1}{a}}(0, 1) \},
\end{aligned}
\]
if $x_0=1$.

Lemma \ref{green} is crucial to prove  the analogues of Theorem \ref{Theorem 3.1} and Theorem \ref{Theorem 3.2} in the strongly degenerate case at least if $x_0=0$ or $x_0=1$; if $x_0 \in (0,1)$ then we refer to \cite[Theorem 2.2]{memoirs}. Indeed we can consider the function $\mathcal L: L^2_{\frac{1}{a}}(0,1) \rightarrow L^2_{\frac{1}{a}}(0,1)$ defined as $\mathcal L(f)= u\in H^2_{\frac{1}{a}}(0,1)$, where $u$ is the unique solution of $\ds \int_0^1 \frac{uv}{a}dx + \int_0^1 u''v'' dx = \ds \int_0^1 \frac{fv}{a}dx$ for all $v \in H^2_{\frac{1}{a}}(0,1)$.

\end{document}